\documentclass[12pt]{article}

\usepackage{amsfonts}
\usepackage{amsmath}
\usepackage{amssymb}
\usepackage{amscd}
\usepackage{amsthm}
\usepackage{indentfirst}
\usepackage[all]{xy}
\usepackage{graphicx,caption}

\newtheorem{thm}{Theorem}[section]
\newtheorem{lemma}[thm]{Lemma}

\newtheorem{rem}[thm]{Remark}

\newcommand{\R}{{\mathbb{R}}}

\newcommand{\Z}{{\mathbb{Z}}}

\newcommand{\C}{{\mathbb{C}}}

\newcommand{\cK}{{\mathcal{K}}}

\newcommand{\cF}{{\mathcal{F}}}

\newcommand{\cL}{{\mathcal{L}}}

\newcommand{\p}{\partial}
\newcommand{\wt}{\widetilde}

\begin{document}

\title{Symplectic quasi-states on the quadric surface and Lagrangian submanifolds\\
}

\renewcommand{\thefootnote}{\alph{footnote}}

\author{\textsc Yakov Eliashberg
and \ Leonid
Polterovich
}


\date{June 12, 2010}

\maketitle

\begin{abstract} The quantum homology of the monotone complex
quadric surface splits into the sum of two fields. We outline a
proof of the following statement: The unities of these fields give
rise to distinct symplectic quasi-states defined by asymptotic
spectral invariants. In fact, these quasi-states turn out to be
``supported" on disjoint Lagrangian submanifolds. Our method
involves a spectral sequence which starts at homology of the loop
space of the 2-sphere and whose higher differentials are computed
via symplectic field theory, in particular with the help of the
Bourgeois-Oancea exact sequence.
\end{abstract}

\renewcommand{\thefootnote}{\arabic{footnote}}

\section{Introduction and main result}
\label{sec-intro}

The quantum homology of the monotone complex
quadric surface $S^2 \times S^2$ splits into the sum of two fields. The
unities of these fields give rise to symplectic
quasi-states defined by asymptotic spectral invariants (see \cite{EP-qst}).
One of these quasi-states is ``supported" on a Lagrangian sphere, the anti-diagonal \cite{EP-rigid}.
Our main finding (see Theorem \ref{thm-main} below) is that the second quasi-state
is ``supported" on the exotic monotone Lagrangian torus described in \cite{EP-rigid}  which is disjoint from the anti-diagonal. Thus these quasi-states are distinct, and the exotic torus
has strong symplectic rigidity properties. Let us pass to precise formulations.

Let $(W,\omega)$ be the standard symplectic quadric surface $S^2
\times S^2$, where both $S^2$ factors have equal areas $1$. Consider
the field $\cK = \C[[t^{-1},t]$ of Laurent series and the quantum
homology algebra $QH(W,\omega)= H_*(W,\C) \otimes \cK$ which is
graded by
$$\text{deg}(at^N) = \text{deg}(a) + 4N\;.$$ The graded component
$QH_4$ is an algebra over $\C$ with respect to the quantum product.
It splits into the sum of two fields generated by idempotents
$e_{\pm}= (1 \pm Pt)/2$, where $1=[W]$ is the fundamental class and
$P$ stands for the class of the point. Define the functionals
$\zeta_\pm : C^{\infty}(W) \to \R$ by
$$\zeta_{\pm}(H) = \lim_{E \to \infty} c(e_\pm,EH)/E\;,$$
where $c(\cdot,\cdot)$ is a spectral invariant
\cite{Oh-spectral,MS}.

It was shown in \cite{EP-qst} that these functionals are {\it
symplectic quasi-states}, that is they are monotone, linear on all
Poisson-commutative subspaces and normalized by $\zeta_\pm(1)=1$.
 Recall that a {\it quasi-measure} associated
to a quasi-state $\zeta$ is a set-function whose value on a closed
subset $X$  equals, roughly speaking, $\zeta(\chi_X)$ where $\chi_X$
is the indicator function of $X$. We write $\tau_\pm$ for the
quasi-measures associated to $\zeta_\pm$. The reader is referred to
\cite{EP-qmm,EP-qst} for further preliminaries.

We view $W$ as the symplectic cut \cite{Lerman} of the unit
cotangent bundle $\{|p| \leq 1\} \subset T^*S^2$, where the zero
section is identified with the Lagrangian anti-diagonal $L \subset
W$, and the length $|p|$ is understood with respect to the standard
spherical metric so that the length of the equator equals $1$. The
level $\{|p|=1/2\}$ contains unique (up to a Hamiltonian isotopy)
monotone Lagrangian torus denoted by $K$. Write $T \subset W$ for
the Lagrangian Clifford torus (the product of equators).

It has been proved in \cite[Example 1.20]{EP-rigid} that
$\tau_- (L)=1$ and $\tau_-(K)=0$. Together with the equality $\tau_-(T)=\tau_+(T)=1$
(see \cite{EP-qst}) this yields that $K$ and $T$ are not Hamiltonian isotopic.

\medskip
\noindent
\begin{thm}\label{thm-main} $\tau_+(L)=0, \tau_+(K) = 1$. In
particular, $\zeta_- \neq \zeta_+$.
\end{thm}

\medskip
\noindent Equality $\tau_+(K) =1$ yields that $K$ is
$e_+$-superheavy in the terminology of \cite{EP-rigid}. In
particular, it is non-displaceable and intersects every image of the
Clifford torus under a symplectomorphism.

The rest of the note contains a detailed outline of the proof of
Theorem \ref{thm-main}. Our method involves a spectral sequence
which starts at homology of the loop space of the 2-sphere and whose
higher differentials are computed via symplectic field theory. The
main technical ingredient comes from a paper by Bourgeois and Oancea
\cite{BO}.

Non-displaceability of the exotic torus $K$ has been recently
established via Lagrangian Floer homology by several independent
groups of researchers:  Fukaya, Oh, Ohta and Ono \cite{FO3},
Chekanov and Schlenk \cite{CH} and Wehrheim (unpublished). The paper
\cite{CH} presents new constructions of exotic Lagrangian tori in
the product of spheres. A related construction of exotic tori is
given by Biran and Cornea in \cite{BC} in the context of their study
of narrow Lagrangian submanifolds. It would be interesting to
understand whether these tori can be distinguished by appropriate
symplectic quasi-measures.

The paper \cite{FO3} contains a more general version of Theorem
\ref{thm-main}. According to \cite{FO3},  the exotic torus $K$ lies
in an infinite family of non-displaceable Lagrangian tori whose
Liouville class varies. Though our approach is quite different from
the one in \cite{FO3}, there are some similarities which deserve
further exploration. Let us mention also that the tori of the
above-mentioned family appear as invariant sets of a semitoric
integrable system (see \cite{PV} and Section \ref{sec-reduction}
below). In dimension four, semitoric means that one of the integrals
generates a Hamiltonian circle action. The study of this class of
integrable systems was initiated recently in \cite{PV}. Semitoric
systems have some amusing properties and appear in meaningful
physical models. It would be interesting to detect ``symplectically
rigid" invariant Lagrangian tori in other examples of semitoric
systems.

\section{Reduction to a Floer-homological
calculation}\label{sec-reduction}

We work with Floer homology with $\C$-coefficients. In our
conventions on the Conley-Zehnder indices, the PSS-isomorphism
identifies $FH_k$ with $QH_{k+2}$ (see \cite{MS} for preliminaries).

Throughout that paper, we denote by $\Sigma$ the diagonal in $W=S^2
\times S^2$. In our picture, $\Sigma$ is obtained from the
hypersurface $\{|p|=1\}$ by the symplectic cut \cite{Lerman}.

Fix $r \in (0;1/2)$, $E>0$ large enough, and $\epsilon >0$ small
enough. We assume that the data is ''non-resonant": $1/r \notin \Z$
and $(E+\epsilon)/\epsilon \notin \Z$.

Consider a Hamiltonian $H_E(|p|)$ on $W$ given by a piece-wise
linear function which equals $E$ on $U:= \{|p| \leq r-\epsilon\}$
and equals $-\epsilon$ on  $V:= \{|p|\geq r\}$.
We
refer e.g. to \cite{EKP} for a discussion on Floer-homological
calculations with piece-wise linear Hamiltonians.

Orbits of period $1$ of $H_E$ form several critical submanifolds,
which we are going to describe now. Each of these submanifolds is
equipped with a Morse function which is used for a small
perturbation of the action functional associated to $H_E$. In
addition, we fix capping discs for orbits from these submanifolds.
Critical points of these Morse functions together with the capping
discs give rise to generators of the Floer complex. Let us pass to
the precise description of this data.

\medskip
\noindent {\sc The maximum set $U$:} Here we have constant orbits
capped with the constant discs. We choose an exhausting Morse
function $f_U$ on $U$ with two critical points: a saddle point $x_0$
of Morse index $2$ and a maximum point $x_2$. Their Conley-Zehnder
indices are $0$ and $2$ respectively, and their actions equal $E$.
We refer to the elements of the Floer complex of the form $\gamma
t^{-N}$ for $\gamma \in \{x_0,x_2\}$ as {\bf $U$-generators}, and
call the number $N$ {\it the $t$-degree}.

\medskip
\noindent{\sc The minimum set $V$:} Here we have constant orbits
capped by the constant discs. We choose an exhausting Morse function
$f_V$ on $V$ with two critical points: a saddle point $y_0$ of Morse
index $2$ and a minimum point $y_{-2}$. Their Conley-Zehnder indices
are $0$ and $-2$ respectively, and their actions equal $-\epsilon$.
We refer to the elements of the Floer complex of the form $\gamma
t^{-N}$ for $\gamma \in \{y_0,y_{-2}\}$ as {\bf $V$-generators}.

\medskip
\noindent {\sc Non-constant orbits:} They form two series of
submanifolds diffeomorphic to $\R P^3$. We denote these submanifolds
by $Z^{\pm}_k$. Here the lower index $k$ corresponds to
$k$-times-covered simple closed geodesics on $L$, $Z^+$ stands for
the orbits on the submanifold $\{|p|= r-\epsilon\}$ and $Z^-$ stands
for the orbits on the submanifold $\{|p|= r\}$. Note that the
multiplicity $k$ satisfies inequalities $k \geq 1$ and
\begin{equation}\label{eq-k-bound}
k \leq \frac{E+\epsilon}{\epsilon}\;.
\end{equation}
In the discussion below we assume that $k \geq 1$ is arbitrary, and
that inequality \eqref{eq-k-bound} is an extra restriction which
selects orbits relevant for the Floer complex corresponding to the
fixed value of $\epsilon>0$. Next, we fix a Morse function
$f^{\pm}_k$ on $Z^{\pm}_k$ with critical points $\check{m}^{\pm}_k$,
$\hat{m}^{\pm}_k$, $\check{M}^{\pm}_k$ and $\hat{M}^{\pm}_k$ of
Morse indices $0,1,2,3$ respectively. It will be convenient for us
to choose $f_k$ in such a way that the orbits in each pair of orbits
$(\hat m_k^\pm,\check m_k^\pm)$ and $(\hat M_k^\pm,\check M_k^\pm)$
represent the same unparameterized orbit. The orbits from
$Z^{\pm}_k$ are capped by discs lying in $W \setminus \Sigma$. We
refer to the elements of the Floer complex of the form $\gamma
t^{-N}$ for $\gamma \in
\check{m}^{+}_k,\hat{m}^{+}_k,\check{M}^{+}_k, \hat{M}^{+}_k$ as
{\bf upper generators} and for $\gamma \in
\check{m}^{-}_k,\hat{m}^{-}_k,\check{M}^{-}_k, \hat{M}^{-}_k$ as
{\bf lower generators}, and call the number $N$ {\it the
$t$-degree}.

The Conley-Zehnder indices of the generators corresponding to the
non-constant orbits are given in Table 1 below.

\begin{table}[h]\label{tab:hresult}
\caption{Indices of upper/lower generators} \centering
\begin{tabular}{c rrrr}
\hline \hline & & & &\\
& $\check{m}_k$ & $\hat{m}_k$ & $\check{M}_k$ & $\hat{M}_k$  \\
upper &  $2k-1$ & $2k$ & $2k+1$ & $2k+2$ \\
lower & $2k-2$  & $2k-1$ & $2k$ & $2k+1$  \\
\hline 
\end{tabular}
\end{table}

The actions of the lower generators equal $-\epsilon + kr$ and of
the upper generators $E+k(r-\epsilon)$.

\medskip\noindent In what follows we write $CZ(\gamma)$ for the Conley-Zehnder index of a
capped orbit $\gamma$ and $A(\gamma)$ for its action. We have
$$CZ(\gamma t^{-N}) = CZ(\gamma) -4N, \; A(\gamma t^{-N}) = A(\gamma) -N\;.$$

A direct calculation (crucially based on the fact that $r < 1/2$)
yields the following lemma which will be used throughout the paper.

\begin{lemma}\label{lem-0}
\begin{itemize}
\item[{(i)}] All lower generators and $V$-generators of the Conley-Zehnder indices $1,2,3$ have action $< 1$.
\item[{(ii)}]All upper generators and $U$-generators of the Conley-Zehnder
indices $1,2,3$ have action $< E+1$ and non-negative $t$-degree.
\item[{(iii}] Let $\gamma$ be an upper generator of
the Conley-Zehnder index $1,2,3$ and action $A(\gamma) > 1$. Then
inequality \eqref{eq-k-bound} holds automatically: $k <
\frac{E+\epsilon}{\epsilon}$.
\item[{(iv)}] There exist numbers $\mu_-(E) < \mu_+(E)$,
$\mu_\pm(E) \to \infty$ as $E \to \infty$ with the following
property: Let $\gamma$ be an upper generator or a $U$-generator of
the Conley-Zehnder index $1,2,3$ and $t$-degree $N$. Then $A(\gamma)
>1$ for  $N \leq \mu_-$ and $A(\gamma) < 1$ for $N > \mu_+ (E)$.
\end{itemize}
\end{lemma}

\begin{proof}

\medskip\noindent{\sc Lower generators:}
Take $\gamma_k \in
\{\check{m}_k^-,\hat{m}_k^-,\check{M}_k^-,\hat{M}_k^-\}$ and put
$\gamma = \gamma_k t^{-N}$. We have that $$CZ(\gamma) = 2k+j-4N \in
[1;3],\;\;j= -2,-1,0,1\;.$$ Thus $0 \leq 2k-4N \leq 4$. Note that
$$A(\gamma) = kr-\epsilon - N\;.$$ Since $r < 1/2$ we have that
$$A(\gamma) < (k/2-N) -\epsilon < 1\;.$$

\medskip\noindent{\sc Upper generators:}
Take $\gamma_k \in
\{\check{m}_k^+,\hat{m}_k^+,\check{M}_k^+,\hat{M}_k^+\}$ and put
$\gamma = \gamma_k t^{-N}$. We have that $$CZ(\gamma) = 2k+j-4N \in
[1;3],\;\;j= -1,0,1,2\;.$$ Thus
\begin{equation} \label{eq-kN}
0 \leq 2k-4N \leq 4\;.
\end{equation}
 In particular, $N \geq 0$ since $k \geq 1$.
Note that
\begin{equation}\label{eq-action}
A(\gamma) = E + kr-k\epsilon -N\;.
\end{equation}
Since $r < 1/2$ we have that
$$A(\gamma) < E + (k/2-N) -k\epsilon \leq E+1\;.$$

Put $\kappa =1-2r+2\epsilon$. Observe that by  \eqref{eq-kN} and
\eqref{eq-action} $A(\gamma) >1$ yields $k < 2E/\kappa <
(E+\epsilon)/\epsilon$ which proves statement (iii) of the lemma.

Further, pick $$\mu_-(E) <(E-1)/\kappa\;, \mu_+(E) >
(E-\kappa)/\kappa\;.$$ Statement (iv) of the lemma readily follows
from  \eqref{eq-kN} and \eqref{eq-action}.

\medskip
\noindent Finally, the only $U$-generator of the index $1,2,3$ is
$x_2$, and its action equals $E$ and its $t$-degree equals $0$. The
only $V$-generator of index  $2$ is $y_{-2}t$ and its action equals
$1-\epsilon$. This completes the proof of statements (i) and (ii) of
the lemma.
\end{proof}

\medskip \noindent \begin{lemma} [Main lemma] \label{lem-1}
$HF_2^{(1;E+1)}(H_E)=\C$.
\end{lemma}

\medskip
\noindent {\bf Proof of Theorem \ref{thm-main} modulo Main Lemma:}

\medskip
\noindent {\sc Step 1:}  Look at the diagram

\[\xymatrix { & HF_3^{(E+1,+\infty)}(H_E) \ar[d]^-{j} &  \\
 HF_2^{(-\infty,1)}(H_E) \ar[r]^-{k} \ar[rd]^-{m} &
 HF_2^{(-\infty,E+1)}(H_E) \ar[d]^-{i} \ar[r]^-{l} & HF_2^{(1,E+1)}(H_E)
 = \C \\ & HF_2^{(-\infty, \infty)} = QH_4(M) = \C^2 &
 }\]

Here the horizontal and the vertical lines are exact, and the
triangle is commutative. Since $e_{\pm}=(1 \pm Pt) /2$ and $\max H_E
= E$  the spectral invariants $c(e_{\pm},H_E)$ do not exceed $E+1$.
Thus, since $QH_4$ is generated by $e_-,e_+$ the map $i$ is onto. By
Lemma \ref{lem-0}(i),(ii) $HF_3^{(E+1,+\infty)}(H_E)=0$. This yields
that $j=0$ so $i$ is an isomorphism, and in particular $l$ has a
non-trivial kernel. Thus $k \neq 0$ and we conclude that $m\neq 0$.

Assume that some non-zero quantum homology class $a=\alpha e_- +
\beta e_+$, $\alpha,\beta \in \C$, lies in the image of $m$.  This
yields $c(a,H_E) \leq 1$. Since $\tau_-(L)=1$ (see \cite{EP-rigid}),
we have that $c(e_-,H_E) \geq E$, and therefore $\beta \neq 0$.
Observe that the quantum product $e_+ * e_-$ equals $0$, while $e_+
* e_+=e_+$, $ e_-
* e_- = e_-$. Thus, by the triangle inequality for spectral
invariants,
$$c(e_+,H_E) = c(a*e_+,H_E) \leq c(a,H_E) + c(e_+,0) \leq 2\;.$$
Since this holds for every $E$ and $r < 1/2$, we conclude that
\begin{equation}\label{eq-tube}
\tau_+(\{|p| < 1/2\})=0\;.
\end{equation}

\medskip
\noindent {\sc Step 2:} Observe now that the Hamiltonian $|p|^2$ on
$W\setminus \Sigma$ extends to a {\it smooth} Hamiltonian on the
whole $W$. This Hamiltonian is integrable and yields a foliation of
$W \setminus (L \cup \Sigma)$ by Lagrangian tori. Look at these tori
in the tube $\{|p| \geq 1/2\}$. One readily checks by an argument in
the spirit of \cite{M}, that all these tori besides the monotone
exotic torus $K$ are displaceable.

One can prove the displaceability directly in the following way.
Write $W=S^2 \times S^2$ as
$$\{x_1^2 +y_1^2 +z_1^2 =1\} \times  \{x_2^2 +y_2^2 +z_2^2 =1\} \subset \R^3 \times \R^3\;.$$
Introduce functions $F$ and $G$ on $W $ by
$$F(x_1,y_1,z_1,x_2,y_2,z_2)= z_1 + z_2\;,$$
$$G(x_1,y_1,z_1,x_2,y_2,z_2)= x_1x_2+y_1y_2+z_1z_2\;,$$
and a map
\begin{equation}
\label{eq-Phi} \Phi= (F,G) : W  \to \R^2\;.\end{equation} One can
check directly that within this model the Hamiltonian $|p|$
corresponds to $\sqrt{(1+G)/2}$. It defines an integrable
Hamiltonian system with an integral $F$ (since $F$ generates a
circle action, such an integrable system is {\it semitoric}
\cite{PV}). The above-mentioned Lagrangian tori are given by
$$N_{a,b}:= \Phi^{-1}(a,b)= \{z_1 +z_2 =a, x_1x_2+y_1y_2+z_1z_2=b\}\;.$$
The monotone torus $K$ is given by $N_{0,-1/2}$.

Note that for $a\neq 0$,  $N_{a,b}$ is displaceable
by
$$(x_1,y_1,z_1,x_2,y_2,z_2) \to (-x_1,y_1,-z_1, -x_2,y_2,-z_2)\;.$$
For $a=0$, the torus $N_{0,b}$, $b \in (-1/2,1)$ can be displaced
inside the hypersurface $\Pi := \{z_1 +z_2 =0\}$. Indeed, let
$\phi_i$ be the polar angle in the $(x_i,y_i)$-plane. Consider a
fibration $\tau: \Pi \to C:=(-1;1) \times S^1$ given by
$$(x_1,y_1,z_1,x_2,y_2,z_2) \to (z_1, \phi_1 -\phi_2)\;.$$
One readily checks that for every $(z,\theta) \in C$ the preimage
$\tau^{-1}(z,\theta)$ consists of a closed orbit of the Hamiltonian
$z_1 +z_2$. Thus for every simple closed curve $\alpha \subset C$,
the preimage $\tau^{-1}(\alpha)$ is a Lagrangian torus in $\Pi$.
Denote by $\sigma$ the push-forward to $C$ of the symplectic form
restricted to $\Pi$. Since the symplectic form on $W$ is given by
$$\frac{1}{4\pi} (dz_1 \wedge d\phi_1 + dz_2 \wedge d\phi_2)\;,$$
we have that $\sigma=(4\pi)^{-1} dz \wedge d\theta$. Furthermore,
$N_{0,b} = \tau^{-1}(\alpha_b)$ with $$\alpha_b = \Big{\{}z^2 =
\frac{\cos \theta -b}{cos \theta +1}\Big{\}}\;.$$ Observe that
$\alpha_b$ is a contractible simple closed curve in $C$. Integration
yields
$$\frac{1}{4\pi}\int_{\alpha_{-1/2}} z d\theta = \frac{1}{2} =  \frac{1}{2} \text{Area}_\sigma (C)\;.$$
For $b > -1/2$ the curve $\alpha_b$ lies inside the disc bounded by
$\alpha_{-1/2}$ in $C$. Thus $\alpha_b$  is displaceable in $C$, and
therefore, by lifting the displacing isotopy to $\Pi$, we get that
$N_{0,b}$ is displaceable in $\Pi$. This completes the proof of the
displaceablility.

\medskip
\noindent {\sc Step 3:} Consider the push-forward $ \Phi_* \tau_+$
of the quasi-measure $\tau_+$ to the plane $\R^2$ by the map $\Phi$
given by \eqref{eq-Phi}. Let $(u,v)$ be Euclidean coordinates on
$\R^2$. Since $F$ and $G$ Poisson commute, $ \Phi_* \tau_+$ extends
to a measure, say $\sigma$ on $\R^2$. Recall that in our model of
$W$ the function $|p|$ corresponds to $\sqrt{(1+G)/2}$, and hence
the tube $\{|p| < 1/2\}$ is given by $\{G < -1/2\}$. Formula
\eqref{eq-tube} above implies that the support of $\sigma$ lies in
$\{v \geq -1/2\}$.

Further, by Step 2, every non-empty fiber $\Phi^{-1}(a,b)$ with
$$ b \geq -1/2, (a,b) \neq (0,-1/2)$$
is displaceable in $W$. Recall \cite{EP-qst} that every
Floer-homological symplectic quasi-measure vanishes on displaceable
subsets, and hence a point $(a,b) \in \R^2$ cannot lie in the
support of $\sigma$ provided the set $\Phi^{-1}(a,b)$ is
displaceable. Therefore the support of $\sigma$ consists of the
single point $(0,-1/2)$, so that $\sigma$ is the Dirac
$\delta$-measure concentrated in this point. Since the torus $K$ is
given by $\Phi^{-1}(0,-1/2)$, we get that
$$\tau_+(K) = \sigma((0,-1/2))=1\;.$$
This completes the proof of the theorem. \qed

\section{Proof of the Main Lemma}

\medskip
\noindent {\sc The strategy of calculation:}

\noindent By Lemma \ref{lem-0}  The Floer complex
$CF_i^{(1;E+1)}(H_E)$, $i=1,2,3$ is generated by upper generators
and $U$-generators in the action window $(1;E+1)$ satisfying the
multiplicity bound \eqref{eq-k-bound}. The lower generators and
$V$-generators leave the stage. {\it Thus we shall suppress the
upper index $+$ and set $Z_k = Z_k^+$, $f_k:=f_k^+,
\check{m}_k=\check{m}_k^+$, etc.}

Denote by $B_n$, $n \geq 0$ the span over $\C$ of generators
$$\gamma \in \{x_0 ,x_2,\check{m}_{k}, \hat{m}_{k}, \check{M}_{k} ,
\hat{M}_{k}, k \geq 1\}$$ of the Conley-Zehnder index $n$.

Put $C_{i,s} = t^{-s}B_{i+4s}$, where $i\geq 0$ and $s \geq 0$.
Write $C_i = \oplus_s C_{i,s}$ and $C = C_1 \oplus C_2 \oplus C_3$.
Denote by $D \subset C$ the subspace consisting of the elements of
action $< 1$ and set $D_{i,s} = D \cap C_{i,s}$. By Lemma
\ref{lem-0}(i)-(iii) the Floer complex of $H_E$ in the action window
$(1;E+1)$ and in degrees $1,2,3$ is given by $(C/D, \delta)$, where
$\delta: C/D \to C/D$ is the Floer differential. The differential
$\delta$ has the form
\begin{equation}\label{eq-delta}
\delta = \delta_0 + t^{-1}\delta_1 + t^{-2} \delta_2+... \; \mod D
\end{equation}
with $\delta_l : B_* \to B_{*+4l-1}$. Let us emphasize that only
negative powers of $t$ appear in this expression: this follows from
the fact that Floer trajectories of $H_E$ are holomorphic near
$\Sigma$ and intersect it positively. With this notation,
\begin{equation}\label{eq-BO-0}
HF_2^{(1;E+1)}(H_E)  = H_2(C/D,\delta)\;.
\end{equation}

\medskip\noindent To motivate the strategy of calculation of this homology group,
identify
\begin{equation}\label{eq-filt}
CF_i^{(1;E+1)}(H_E)= CF_i^{(1-E;1)}(H_E-E)\;
\end{equation}
and look at  $C_{i,0}/D_{i,0}$ considered as a subspace of
$CF_i^{(1-E;1)}(H_E-E)$. With this identification homology of the
complex $(\oplus C_{i,0}/D_{i,0}, \delta_0)$ converge to symplectic
homology $SH(U')$ of the domain $U'=\{|p| < r\}$ in the action
window $(-\infty; 1)$ provided $E \to \infty$ and $\epsilon \to 0$.
Indeed, functions $H_E-E$ restricted to $U'$ form an exhausting
sequence used in the definition of symplectic homology, the complex
$\oplus C_{i,0}/D_{i,0}$ is generated by closed orbits of $H_E-E$
capped inside $U'$ and the differential $\delta_0$ counts the Floer
trajectories lying inside $U'$. The contribution of the lower
generators disappears in this limit.

Now we can formulate the intuitive idea behind our calculation: The
complex $(C/D,\delta)$ can be considered as a properly understood
deformation of $(\oplus C_{i,0}/D_{i,0}, \delta_0)$ which involves
capping discs and Floer trajectories intersecting $\Sigma$.
Eventually, the required homology $H(C/D,\delta)$ can be calculated
by an appropriate spectral sequence which starts at $SH(U')$.

To make this precise, we use the technology developed by Bourgeois
and Oancea \cite{BO} (and extended further in \cite {BEE})  who
identified symplectic homology of the Liouville domain $\{|p| < r\}$
with the homology of the complex $B: = \oplus B_i = \oplus_i
C_{i,0}$ equipped with certain differential $d_0$ which will be
described below. In fact we shall introduce an appropriate
deformation of their construction which takes into account the fact
that Floer cylinders can intersect $\Sigma$ and which will enable us
to calculate homology of the deformed complex $(C/D,\delta)$.

As a graded and filtered \footnote{ One should shift our filtration
by $E$ to get the standard filtration on symplectic homology used in
\cite{BO}, see \eqref{eq-filt}.} vector space the deformed
Bourgeois-Oancea complex is described as follows.
 Introduce the ring $\Lambda$
consisting of all Laurent series of the form
$$\sum_{s=0}^{+\infty} \lambda_s t^{-s}\;, \lambda_s \in \C \;.$$
With this notation the deformed Bourgeois-Oancea complex is given by
$QB:= B \otimes_{\C} \Lambda$. As before, this complex is graded by
$CZ(\gamma t^{-s}) = CZ(\gamma)-4s$ and filtered by the symplectic
action of $H_E$: $A(\gamma t^{-s}) = A(\gamma)-s$. Its differential
$d$ is $\Lambda$-linear and has the form
$$d = d_0 + t^{-1}d_1 + t^{-2}d_2 +...\;,$$
with $d_l : B_* \to B_{*+4l-1}$. It is instructive to view $(QB,d)$
as a ``quantum" deformation of the complex $(B,d_0)$ where $t$ plays
the role of a deformation parameter. By \cite{BO} the group
$H(B,d_0)$ coincides with symplectic homology of the Liouville
domain $\{|p| < r\} \subset T^*S^2$. The latter, according to
\cite{AS,SW}, equals homology of the free loop space of $S^2 = L$.
Therefore

\begin{equation}
\label{eq-loops} H(B,d_0) = H(\cL S^2)\;.
\end{equation}

In order to describe the differential $d$ we need some
preliminaries.

\medskip
\noindent{\sc Stretching-the-neck:}

\noindent Denote by $\pi: \nu \to \Sigma$ the holomorphic normal
line bundle to $\Sigma$ in $W= \C P^1 \times \C P^1$. Perform a
stretching-the-neck \cite{EGH,B-etc} of $W$ at the hypersurfaces
$\{|p|= r-\epsilon\}$ and $\{|p| = r\}$. The manifold $W$ splits
into three pieces which after gluing in (asymptotically) cylindrical
ends at their boundaries will be identified with $W_{left} := W
\setminus \Sigma$, $W_{middle}: = \nu \setminus \Sigma$ and
$W_{right}:= \nu$.

It would be convenient to view orbits from $Z_k$ as
$k$-times-covered unit circles of the bundle $\nu$. In particular,
the projection $\pi: \nu \to \Sigma$ gives rise to the natural map
$\pi^k: Z_k \to \Sigma$.

We shall assume that the exhausting Morse function $f_U$ is defined
on the whole $W_{left}$.

\medskip\noindent {\sc Matrix coefficient $(d_l\gamma_+,\gamma_-)$ for upper
generators $\gamma_+,\gamma_-$:}

\noindent Denote by $P_{a,b}$   the cylinder $\C \setminus {0}$ with
the set of {\it negative} interior punctures $a=\{a_1,...,a_{l_-}\}$
and the set of {\it positive} interior punctures $b=\{b_1, ...,
b_{l_+}\}$. Here we put $l_+ =l$ (recall that we are defining
$d_l$).

Let $\gamma$ be an orbit from $Z_k$. In our picture it is
interpreted in two different ways. First, it is a point of the
corresponding critical variety $Z_k$ (recall that upper index $+$ is
omitted). We denote this point by $A_\gamma$. Second, $\gamma$ is a
(parameterized, in general multiply covered) unit circle in the
fiber of the bundle $\nu$ over the projection $\pi(A_\gamma)$.

In what follows we work with holomorphic maps $$u : P_{a,b} \to
W_{middle}= \nu \setminus \Sigma \;.$$  We say that $u$ {\it enters}
$\gamma$ at a puncture $\eta \in a \cup \infty$ if $u(z)/|u(z)| \to
\gamma(\text{arg}(z))$ and $|u(z)| \to \infty$  as $z \to \eta$. We
say that $u$ {\it exits } $\gamma$ at a puncture $\eta \in b \cup 0$
if $u(z)/|u(z)| \to \gamma(\text{arg}(z))$ and $|u(z)| \to 0$  as $z
\to \eta$.

Suppose that $u$ exits $\alpha$ at $0$, enters $\beta$ at $\infty$ and
in addition exits (resp. enters) {\it some} orbits at all positive (resp. negative)
interior punctures. We shall denote this by
$$\xymatrix{\alpha \ar[r]^u & \beta}\;.$$ We shall also book-keep the quantity
$l_-$ by putting \begin{equation}\label{eq-weight}\text{weight}(u) =
2^{l_-}\;.\end{equation}

Note that geometrically such $u$'s either are contained in a single
fiber of $\nu$, or correspond to multi-sections of $\nu$ with zeroes
at $\pi(A_\alpha)$ and the projections of the asymptotic images of
the positive punctures, and with poles at $\pi(A_\beta)$ and the
projections of the asymptotic images of the negative punctures.

Suppose now that $\alpha,\beta$ belong to the same critical manifold
$Z_k$. We shall write
\begin{equation}
\label{eq-arrow} \xymatrix{\alpha \ar@{~>}[r]^v & \beta}\;
\end{equation}
 if $v$ is
a parameterized piece of trajectory of the negative gradient flow of
$f_k$ joining the points $A_{\alpha}$ and $A_{\beta}$.
\footnote{We will assume that the gradient vector field for each function $f_k$ satisfies the following    condition:
 the $1$-dimensional stable manifold    of  $\check  M_k$ (resp. the unstable manifold of $\hat m_k$)
 consists of orbits  which differ from $\check M_k$ (resp. from $\hat m_k$) only by their parameterization.}

Let $\gamma_+,\gamma_-$ be two upper generators representing
critical points of $f_{k_+},f_{k_-}$ on $Z_{k_+},Z_{k_-}$
respectively. Assume that $k_+ \neq k_-$. Consider all possible
configurations of the form
\begin{equation}\label{eq-BOconfig}
\xymatrix {\gamma_+ \ar@{~>}[r]^v & \alpha \ar[r]^u &  \beta
\ar@{~>}[r]^w & \gamma_- }\;. \end{equation} We call
$\text{weight}(u)$ the weight of this configuration. Note that the
right and/or the left arrow could be empty. In case $k_+ =k_-=k$, we
consider configurations of the form
\begin{equation}
\label{eq-arrow-1} \xymatrix{\gamma_+ \ar@{~>}[r]^v & \gamma_-}\;,
\end{equation}
and its weight is put to be $1$. Finally, we define
$(d_l\gamma_+,\gamma_-)$ as the sum of weights taken over the
$0$-dimensional components of the moduli spaces of configurations of
the form \eqref{eq-BOconfig} and  \eqref{eq-arrow-1}. Note that in
the definition of the moduli spaces \eqref{eq-BOconfig} the markers
$a,b$ are varying as well. In addition, each weight should be taken
with a sign responsible for the orientation of the moduli space. The
orientation issue will be ignored in this note.

\bigskip
\noindent

\noindent {\sc Matrix coefficients $(d_l x, \gamma )$ and $(d_l
\gamma,x)$ for an upper orbit $\gamma$ and an $U$-orbit $x$:}

\noindent First, note that all the coefficients $(d_l \gamma, x)$
vanish for $l > 0$ by index reasons, so we do not need to describe
here the algorithm for their computing. For the description of
$(d_0\gamma, x)$ we refer to \cite{BEE}.

Let us describe the algorithm for computing of coefficients $(d_l x,
\gamma )$. Suppose that
 the multiplicity of the orbit $\gamma$ is equal to $k$.
 Then the coefficient $(d_lx,\gamma)$ counts rigid configuration
\begin{equation}
\label{eq-guc} (g,u,c)\; ,
\end{equation}
 where
\begin{itemize} \item
$g$ is a  minus gradient trajectory of  the function $f_U$ which
begins
 at $x$ and ends at a point $p\in \p U$; note that $\p U$ can be canonically
  identified with the $S^1$-bundle  associated with the complex line
   bundle $\nu$, and thus $p$ determines a ray $l_p$ in one of the fibers of $\nu$;
\item $u:P_{a,b}\to W_{middle}$ a holomorphic map such that $u$ exits
 (resp. enters)  some orbits at  all positive (resp. negative) punctures,
  enters  an orbit $\wt \gamma$ of multiplicity $k$ at $\infty$
   and $\lim\limits_{z\to 0} u(z) \in l_p$; note that the set of
   positive
    punctures must be non-empty due to the maximum principle;
\item  $c$ is a minus gradient trajectory of the function $f_k$
 connecting $\wt \gamma$ and $\gamma$. \end{itemize}

   A new feature of
     configurations $(g,u,c)$ considered above is the ray $l_p$
which connects the holomorphic curve $u$ with the gradient
trajectory $g$. This requires a justification which will be given
elsewhere.

\bigskip
\noindent {\it This completes the description of the differential
$d$ on $QB$.}

\bigskip

\medskip
\noindent {\sc Comparison of Floer and Bourgeois-Oancea homologies:}

\noindent Recall that $D$ denotes the subspace of $C$ consisting of
elements of action $< 1$. Denote by $QD \subset QB$ the subspace
consisting of elements of action $< 1$. We shall use the following
equality:
\begin{equation}\label{eq-BO}
H_2(C/D, \delta) = H_2(QB/QD,d)\;.
\end{equation}

Note that for $i=1,2,3$ we have $C_i = QB_i = \oplus _s B_{i+4s}
\otimes_\C t^{-s}$, and $D_i = QD_i = QB_i \cap D$. Let us compare
the differentials. The stretching-the-neck procedure described above
has the following effect on the original Floer trajectories of our
Hamiltonian $H_E$: every Floer trajectory joining a pair of upper
generators $\gamma_+$ and $\gamma_-$ splits into three pieces. The
piece lying in $W_{middle} = \nu \setminus \Sigma$ is the Floer
trajectory joining $\gamma_+$ and $\gamma_-$ with positive punctures
(corresponding to the intersection points with $\Sigma$) and
possibly some negative punctures. The orbit $\gamma_+$ lies on the
connected component of the ideal boundary of $W_{middle}$ adjacent
to $W_{right}$, while the orbit $\gamma_-$ lies on the connected
component of the ideal boundary of $W_{middle}$ adjacent to
$W_{left}$. The positive punctures are capped by rigid holomorphic
planes (the fibers of $\nu$) lying in $W_{right}$.  The negative
punctures are capped by rigid holomorphic planes lying in $W_{left}$
corresponding to $\C P^1 \times \text{point}$ and $\text{point}
\times \C P^1$ in $\C P^1 \times \C P^1$.  Thus every negative
puncture is capped by exactly $2$ rigid planes, which yields the
factor $2^{l_-}$ in the definition of the weight in
\eqref{eq-weight}.

Let's focus on the piece lying in $W_{middle}$: one extends
Bourgeois-Oancea theory \cite[Prop. 4]{BO} and finds an isomorphism
between homology whose differential is determined by such punctured
Floer trajectories and the homology whose differential is described
by configurations of the form \eqref{eq-BOconfig} and
\eqref{eq-arrow-1}. This explains equality \eqref{eq-BO}. The formal
proof will be given in a forthcoming paper.

\medskip
\noindent\begin{rem}{\rm From the viewpoint of the Hamiltonian Floer
theory, the differential $\delta: C/D \to C/D$ does not lifts
canonically to a differential $C \to C$: the square of the
expression in the right hand side of formula \eqref{eq-delta}
vanishes modulo the subcomplex $D$. Since $d^2=0$, the argument
above shows that in the limit, ``when the neck is stretched", the
square of this expression vanishes in $C$ itself.}
\end{rem}

\medskip
\noindent {\sc Unperturbed differential:}

\noindent The explicit form of the differential $d_0$ is folkloric
(private communications of F.Bourgeois and T.Ekholm). We shall
present the result right now.

\medskip\noindent
\begin{lemma}\label{lem-folk}
We have
$$d_0\hat{m}_k = 0, d_0\hat{M}_k=0\;, k \geq 1\;,$$
$$ d_0\check{m}_k = 2\hat{m}_{k-1}+2\hat{M}_{k-2}, k \geq 3\;,$$
$$ d_0 \check{M}_k = 2\hat{m}_k + 2\hat{M}_{k-1}, k \geq 2\;,$$
$$ d_0\check{m}_2 = 2\hat{m}_{1}+2 x_2\;,$$
$$ d_0 \check{M}_1 = 2\hat{m}_1 + 2x_2\;,$$
$$d_0 \check{m}_1 =0\;,$$
$$d_0 x_0 =0, d_0 x_2 = 0\;.$$
\end{lemma}

\medskip
\noindent It follows that the homology of the complex $(B,d_0)$ are
given by
$$H_0(B,d_0) = \text{Span}_{\C} ([x_0]), H_2(B,d_0) = \text{Span}_{\C}
([x_2])\;,$$ $$H_{2k+2}(B,d_0) = \text{Span}_{\C} ([\hat{M_k}])\;,
 k \geq 1,$$
$$H_1(B,d_0)= \text{Span}_{\C} ([\check{m}_1])\;,$$
$$H_{2k+1}(B,d_0)= \text{Span}_{\C} ([\check{M}_k-\check{m}_{k+1}])\;,  k \geq 1.$$

\medskip\noindent
It remains to describe the differential $d_l$ for $l \geq 1$.

\medskip
\noindent {\sc Calculation of the quantum corrections:}

\medskip
\noindent \begin{lemma}\label{lem-2}
\begin{itemize}
\item[{(i)}] For all $l \geq 2$ one has $d_l = 0$;
\item[{(ii)}] For all $k \geq 1$,
$$d_1 x_0 = d_1 x_2 = 0$$
$$d_1 \check{m}_k = \hat{m}_{k+1},\;
d_1\check{M}_k=\hat{M}_{k+1},\;d_1 \hat{m}_k=0, d_1 \hat{M}_k=0\;.$$
\end{itemize}
\end{lemma}

\medskip
\noindent {\bf Proof of Main Lemma:}

1) Define a decreasing filtration $\dots \supset QB^{(\mu)} \supset
QB^{(\mu+1)} \supset \dots $ on $QB$ by
$$QB_*^{(\mu)} = \oplus_{s \geq \mu+1}C_{*,s} = \oplus_{s \geq \mu+1}t^{-s}B_{*+4s}\;,$$
and observe that the differential $d$ preserves the filtration.
Furthermore, by Lemma \ref{lem-0}(iv) for $E$ large enough the
subspace $QD_i \subset QB_i$, $i=1,2,3$ consisting of elements of
the $H_E$-action $< 1$ is squeezed between $QB_i^{(\mu_+)}$ and
$QB_i^{(\mu_-)}$ with $\mu_\pm(E) \to \infty$ as $E \to \infty$:
$$QB_i^{(\mu_-)} \supset QD_i \supset QB_i^{(\mu_+)}\;.$$
We shall show in the next step that for $\mu$ large enough
$H_2(QB/QB^{(\mu)},d)$ is isomorphic to $\C$ and is generated by
$x_2 \mod QB^{(\mu)} $. The latter implies that the map
$H_2(QB/QB^{(\mu_+)},d) \to  H_2(QB/QB^{(\mu_-)},d)$ is an
isomorphism. Since it factors through $H_2(QB/QD,d)$ we shall
conclude that $H_2(QB/QD,d)=\C$.

2) Lemma \ref{lem-2} yields that $d_1: H_{2k}(B,d_0) \to
H_{2k+3}(B,d_0)$ vanishes while $d_1: H_{2k+1}(B,d_0) \to
H_{2k+4}(B,d_0)$ is onto for all $k \geq 0$. Fix $\mu$ large enough
and identify $\overline{QB}:= QB/QB^{(\mu)}$ with $\oplus_{s \leq
\mu} C_{*,s}$.

We claim that
$$H_2(\overline{QB},d) = H_2(B,d_0) = \C\;.$$
Indeed, consider a filtration $\cF_p \overline{QB}_*:= \oplus_{\mu
\geq s \geq \mu-p} C_{*,s}$ on the complex $\overline{QB}$. Look at
the homology spectral sequence corresponding to this filtration
\cite{Spanier}: we have that
$$E^1_j = \oplus_{s \leq \mu} H_j(C_{*,s}, d_0) = \oplus_{s \leq \mu}
t^{-s}H_{j+4s}(B,d_0)$$ and the differential $d^1: E^1_j \to
E^1_{j-1}$  can be written as
$$d^1= \oplus d^1_s,\;\, d^1_s = [t^{-1}d_1]: H_j( C_{*,s}, d_0) \to H_{j-1}(C_{*,s+1}, d_0)\;, s \leq \mu-1\;.$$
Since $d=d_0 + t^{-1}d_1$, the spectral sequence degenerates at the
second page and converges to $E^2_j := H_j(E^1, d^1)$, and in
particular $H_2(\overline{QB},d) = E^2_2$. In order to calculate
$E^2_2$ look at the piece
\begin{equation} \label{eq-piece}
H_3(C_{*,s-1}, d_0) \to H_2(C_{*,s}, d_0) \to H_1(C_{*,s+1}, d_0)
\end{equation}
of the complex $(E^1, d^1)$. If $\mu-1 \geq s \geq 1$, the left arrow is onto,
while the right arrow is zero, thus the homology vanishes.
If $s=\mu$, the sequence \eqref{eq-piece} degenerates to
$$H_3(C_{*,\mu-1}, d_0) \to H_2(C_{*,\mu}, d_0) \to 0\;,$$
and since the left arrow is onto, the homology vanishes. Finally,
for $s=1$ the sequence \eqref{eq-piece} degenerates to
$$0 \to H_2(C_{*,0}, d_0) \to H_1(C_{*,1}, d_0)\;.$$
Since the right arrow vanishes, the resulting homology is
$$H_2(C_{*,0},d_0) = H_2(B,d_0)= \C\;.$$ We conclude that $H_2(\overline{QB},d) = E^2_2
=\C$, and the generator of $H_2(\overline{QB},d)$ is $x_2 \mod
QB^{(\mu)}$, as required.

\medskip\noindent
It remains to put all the pieces together: The calculation above
together with equalities \eqref{eq-BO-0} and \eqref{eq-BO} yield

$$ HF_2^{(1;E+1)}(H_E)  = H_2(C/D,\delta)=  H_2(QB/QD,d)= H_2(B,d_0)=\C\;.$$
This completes the proof.
\qed

\medskip
\noindent
{\bf Proof of Lemma \ref{lem-2}:}

1) First we explore coefficients of the form
$(d_l \gamma_+,\gamma_-)$ with $l=l_+ \geq 1$,
 $$\gamma_\pm \in \{\check{m}_{k_\pm},
\hat{m}_{k_\pm}, \check{M}_{k_\pm} , \hat{M}_{k_\pm}\}\;.$$
 The index formula reads
$$CZ(\gamma_+)-CZ(\gamma_-) + 4l_+ = 1\;.$$
Recall that $$CZ(\gamma_\pm) = 2k_\pm + j_\pm, \; j_\pm =
-1,0,1,2\;.$$ Put $h=j_+ -j_-$. Thus the index formula yields
\begin{equation}\label{eq-index}
2(k_+-k_-) + h +4l_+ =1\;.
\end{equation}

Consider a configuration of the form
$$\xymatrix {\gamma_+ \ar@{~>}[r]^v & \alpha \ar[r]^u &  \beta
\ar@{~>}[r]^w & \gamma_- }\;.$$ Denote by $\Delta$ the degree of the
projection of $u$ to $\Sigma$. Since the Chern class of $\nu$ equals
$2$ we have
\begin{equation}\label{eq-RR}
(k_+ + l_+) - (k_-+l_-) = 2\Delta\;.
\end{equation}

It follows from \eqref{eq-index} and \eqref{eq-RR} that
$$l_+ + l_- = (1-h)/2-2\Delta\;.$$
Since $l_+ \geq 1$, $\Delta \geq 0$ and $h$ is an integer from
$[-3;3]$ we conclude that $\Delta=0$. Thus our holomorphic curve
lies in the single fiber of the bundle $\nu$.

We claim that $h \neq -3$. Indeed otherwise we have the connecting
trajectory of the form
\begin{equation}\label{eq-mM}
\xymatrix {\check{m}_{k_+}  \ar@{~>}[r]^v & \alpha \ar[r]^u & \beta
\ar@{~>}[r]^w & \hat{M}_{k_-} } \;. \end{equation} Since
$\check{m}_{k_+}$ is the point of minimum of $f_{k_+}$, $v$ is the
constant trajectory. Since $\hat{M}_{k_-}$ is the point of maximum
of $f_{k_-}$, $w$ is the constant trajectory. Since $u$ lies in the
single fiber, the set of limit points of $u(z)/|u(z)|$ as $z \to
\infty$ lie on a circle which is the fiber of $Z_{k_{-}} =\R P^3$
over $\pi^{k_+}(\check{m}_{k_+})$. Generically, as the dimension
count shows, this circle does not pass through $\hat{M}_{k_-}$, and
hence configuration \eqref{eq-mM} does not exist. The claim follows.

Since $h \neq -3$, we get that
$$h=-1, l_+=1, l_- =0\;.$$
This readily yields that $d_l \gamma_+ = 0$ for $l \geq 2$ and the
only possible non-trivial matrix coefficients could be  (with
$k_+:=k$) $(d_1 \check{m}_k,\hat{m}_{k+1})$, $(d_1
\check{M_k},\hat{M}_{k+1})$ and $(d_1\hat{m}_k,\check{M}_{k+1})$. We
claim that $$(d_1 \check{m}_k,\hat{m}_{k+1})=(d_1
\check{M_k},\hat{M}_{k+1})=1\;.$$

Let us present a calculation (modulo orientations) for the
coefficient $(d_1 \check{m}_k,\hat{m}_{k+1})$ (the calculation for
$(d_1 \check{M_k},\hat{M}_{k+1})$ is analogous). Since $\Delta=0$
and $l_+=1$, we work with holomorphic maps $u$ of the cylinder
$P_{\{b\}}$ punctured at a point $b \in \C \setminus 0$ lying in the
single fiber of $\nu$. Choose functions $f_k$ and $f_{k+1}$ on $Z_k$
and $Z_{k+1}$ respectively so that
$$\pi^k(\check{m}_k)=\pi^{k+1}(\hat{m}_{k+1}) := m \in \Sigma$$ and so that the  circle
$(\pi^{k+1})^{-1}(m) \subset Z_{k+1}$ forms the unstable manifold
$\mathcal{U}$ of $\hat{m}_{k+1}$. Identify the fiber of the line
bundle $\nu$ over $m$ with $\C$. Recall that we identified each
$Z_j$ with the unit circle bundle of $\nu$. Assume that
$\check{m}_k$ corresponds to the point $-1 \in \C$ and
$\hat{m}_{k+1}$ corresponds to the point $1 \in \C$. Since
$\check{m}_k$ is the minimum point of $f_k$, the only gradient
trajectories exiting $\check{m}_k$ are the constant ones. Thus we
are counting configurations of the form
\begin{equation}\label{eq-mm}
\xymatrix {\check{m}_{k} \ar[r]^u &  \beta \ar@{~>}[r]^w &
\hat{m}_{k+1} } \;. \end{equation} Note that the point $A_\beta$
lies both on the circle $\mathcal{U}= (\pi^{k+1})^{-1}(m)$ (since
the image of $u$ is contained in the single fiber) and on the stable
manifold of $\hat{m}_{k+1}$. These two submanifolds intersect
transversally at a single point, $\hat{m}_{k+1}$. In particular, $w$
is constant. Therefore it suffices to show that the holomorphic map
$u$ is unique up to a reparameterization and up to multiplication by
constants (these symmetries should be taken into account when one
passes to the moduli space).

The general form of $u$ is $u(z) = \lambda z^k (z-b)$ with the
asymptotic conditions
$$\text{Arg}(u(t)) \to 0 \;,\; t \in \R_+, t \to +\infty$$
and
$$\text{Arg}(u(t)) \to \pi \;,\; t \in \R_+, t \to 0\;.$$
This readily yields $\lambda, b \in \R_+$. The change of variables
$z \to cz$ with $c \in \R_+$ takes $u$ to the form $u(z) = \lambda
c^{k+1}z^k(z-b/c)$. Thus $u$, up to a reparameterization and up to
multiplication by constants coincides with $z^k(z-1)$. Thus we have
the unique configuration of the form \eqref{eq-mm}, which completes
the calculation.

\medskip

Finally, we claim that $(d_1\hat{m}_k,\check{M}_{k+1})=0$. Indeed,
assume that $d_1\hat{m}_k = n\check{M}_{k+1}$. Observe that since
$d^2=0$ we have that $d_1d_0+d_0d_1 =0$. Thus (in view of the
explicit formulas for the unperturbed differential)
$$d_1d_0\hat{m}_k + d_0d_1\hat{m}_k = 0 + n
d_0\check{M}_{k+1}=0\;.$$ Since $d_0\check{M}_{k+1}\neq 0$ we
conclude that $n=0$ as claimed.

\medskip

2) Now we turn to calculation of $d_l x$ where $x \in \{x_0,x_2\}$
and $l \geq 1$. Let us observe  that  if  $(d_lx,\gamma)\neq 0$ then
the orbit $\gamma$ must have an odd grading, and hence
$\gamma=\check m_k$, or $\gamma=\check M_k$.
 But in this case  if there exists a gradient  trajectory $c$ connecting an orbit $\wt \gamma$ with $\gamma$, and  if
$\wt\gamma_\alpha$ differs  from $\wt\gamma$    by a
reparameterization $s\mapsto se^{i\alpha}$,
 then for almost all values $\alpha$ there exists a gradient trajectory $c_\alpha$ connecting  $\wt\gamma_\alpha $
    with $\gamma$.
  This implies that there  are no rigid configurations $(g,u,c)$ (see \eqref{eq-guc} above) which may contribute to
   $(d_lx,\gamma)$, and hence $(d_lx,\gamma)=0$. Indeed, any such configuration
    belongs to a family $(g,u\circ r_\alpha,c_\alpha)$, where $r_\alpha:\C\setminus 0\to\C\setminus 0$ is
     the rotation $z\mapsto ze^{i\alpha}$.

 This  finishes off the proof of the lemma. \qed

\medskip\noindent{\bf Outline of the proof of Lemma \ref{lem-folk}:}
First we discuss coefficients of the form
$(d_0 \gamma_+,\gamma_-)$ with $l_+ =0$,
 $$\gamma_\pm \in \{\check{m}_{k_\pm},
\hat{m}_{k_\pm}, \check{M}_{k_\pm} , \hat{M}_{k_\pm}\}\;.$$ We argue
as in Step 1 of the proof of Lemma \ref{lem-2}, keeping the same
notations and taking into account that $l_+ =0$ This yields three
possibilities:
\begin{itemize}
\item[{(i)}] $h=-3, \Delta=0, l_- =2$;
\item[{(ii)}] $h=-3, \Delta=1, l_-=0$;
\item[{(iii)}] $h=-1, \Delta=0, l_-=1$.
\end{itemize}
Case (i) is ruled out exactly as in the proof of Lemma \ref{lem-2}.
Case (ii) yields $(d_0 \check{m}_k,\hat{M}_{k-2}) =2$,
case (iii) yields
$$(d_0 \check{m}_k, \hat{m}_{k-1})= (d_0 \check{M}_k, \hat{M}_{k-1}) =2\;,$$
while all other coefficients vanish.

Further, we analyze the matrix coefficients involving $x_0$ and $x_2$.
Equality $(d_0\check{m}_2,x_2) = 2$ follows from the count of degree $1$ (properly parameterized) sections of the bundle
$\nu$ passing through a given generic point and having a single zero of order two
at the point $\pi^2(\check{m}_2) \in \Sigma$.
Equality
$(d_0\check{M}_1,x_2) = 2$ corresponds to the fact that there are exactly two rigid spheres
$S^2 \times \text{point}$ and $\text{point} \times S^2$ in $W$ passing through $x_2$.

The only tricky remaining coefficient is $(d_0 \check{m}_1,x_0)=0$:
Seemingly, there are two rigid spheres in $W$ asymptotic to
$\check{m}_1$ which may contribute to this coefficient. We claim
that in fact a cancelation happens. To see this, recall that by
\eqref{eq-loops} and \cite{Cohen} $H_1(B,d_0) = \Z$. If $(d_0
\check{m}_1,x_0) \neq 0$, we would get that $H_1(B,d_0) = 0$, and
thus arrive at a contradiction. \qed

 \bigskip \noindent {\bf Acknowledgements.} We thank Luis Diogo, Misha Entov, Sam Lisi, Dusa
McDuff, Isidora Milin, Yasha Savelyev and Frol Zapolsky for numerous
useful discussions. Preliminary results of this note were announced
at the PRIMA congress in Sydney in July 2009. The second named
author thanks the Simons Foundation for sponsoring his stay at MSRI,
Berkeley, where a part of this paper has been written, and MSRI for
hospitality and a stimulating research atmosphere.

\bibliographystyle{alpha}

\begin{thebibliography}{99}

\bibitem{AS} Abbondandolo, A., Schwarz, M.,
{\it On the Floer homology of cotangent bundles,} Comm. Pure Appl.
Math. {\bf 59} (2006), 254--316.

\bibitem{B-etc}  Bourgeois, F., Eliashberg, Y., Hofer, H., Wysocki, K., Zehnder,
E., {\it Compactness results in symplectic field theory,} Geom.
Topol. {\bf 7} (2003), 799--888.

\bibitem{BEE} Bourgeois, F., Ekholm, T., Eliashberg, Y., {\it Effect of Legendrian
Surgery,} preprint  arXiv:0911.0026, 2009.

\bibitem{BC} Biran, P., Cornea, O., {\it Rigidity and uniruling for Lagrangian submanifolds,}
Geom. Topol. {\bf 13} (2009), 2881--2989.

\bibitem{BO} Bourgeois, F.,  Oancea, A., {\it An exact sequence for
contact and symplectic homology},  Invent. Math. {\bf 175} (2009),
no. 3, 611--680.

\bibitem{CH} Chekanov, Y., Schlenk, F., {\it Notes on monotone Lagrangian twist tori,}
preprint arXiv:1003.5960, 2010.

\bibitem{Cohen} Cohen, R., Jones, J.,Yan, J., {\it The loop homology algebra of spheres and projective spaces,} in {\it Categorical decomposition techniques in algebraic topology (Isle of Skye, 2001)},  77--92, Progr. Math., 215, Birkhäuser, Basel, 2004.

\bibitem{EGH}  Eliashberg, Y., Givental, A., Hofer, H., {\it Introduction to symplectic field
theory, } Geom. Funct. Anal. 2000, Special Volume, Part II,
560--673.

\bibitem{EKP} Eliashberg, Y., Kim, S.-S., Polterovich, L., {\it Geometry
of contact transformations and domains: orderability versus
squeezing}, Geom. Topol. {\bf 10} (2006) 1635-1747.


\bibitem{EP-qmm} Entov, M., Polterovich, L., {\it Calabi quasimorphism
and quantum homology, } Intern. Math. Res. Notices {\bf 30} (2003),
1635-1676.

\bibitem{EP-qst} Entov, M., Polterovich, L.,
{\it Quasi-states and symplectic intersections, } Comm. Math. Helv.
{\bf 81}:1 (2006), 75-99.

\bibitem{EP-rigid} Entov, M., Polterovich, L., {\it
 Rigid subsets of symplectic manifolds}, Compos. Math. {\bf 145} (2009), no. 3, 773--826.

\bibitem{FO3} Fukaya, K., Oh. Y.-G., Ohta, H., Ono, K., {\it Toric degeneration
and non-displaceable Lagrangian tori in $S^2 \times S^2$}, preprint
arXiv:1002.1660, 2010.

\bibitem{Lerman}  Lerman, E., {\it Symplectic cuts}, Mathematical Research Letters {\bf 2} (1995), 247–258

\bibitem{M} McDuff, D., {\it  Displacing Lagrangian toric fibers via
probes}, preprint arXiv:0904.1686, 2009.

\bibitem{MS} McDuff, D., Salamon, D., {\it
$J$-holomorphic curves and symplectic topology,} AMS, 2004.

\bibitem{Oh-spectral} Oh, Y.-G.,
{\it Construction of spectral invariants of Hamiltonian
diffeomorphisms on general symplectic manifolds, } in {\sl The
breadth of symplectic and Poisson geometry, 525-570,} Birkh\"auser,
Boston, 2005.

\bibitem{PV} Pelayo, A., V\~u Ng\d{o}c, S., {\it Semitoric integrable systems on symplectic
 4-manifolds, } Invent. Math. {\bf 177} (2009), 571--597.

\bibitem{Spanier} Spanier, E.H., {\it  Algebraic topology,} Springer-Verlag, New York-Berlin, 1981.

\bibitem{SW} Salamon, D. A., Weber, J., {\it Floer homology and the heat flow},
Geom. Funct. Anal. {\bf 16} (2006), 1050--1138.


\end{thebibliography}

\bigskip

\noindent
\begin{tabular}{ll}
Yakov Eliashberg & Leonid Polterovich\\
Stanford University &  University of Chicago and Tel Aviv University\\
eliash@math.stanford.edu & polterov@runbox.com \\
\end{tabular}

\end{document}